\documentclass[final,3p,times]{elsarticle}

\usepackage{amsmath, amsthm, amssymb}
\usepackage{amsfonts}
\usepackage{verbatim}
\usepackage{color}
\numberwithin{equation}{section}
\usepackage{hyperref}

\newtheorem{theorem}{Theorem}

\theoremstyle{definition}

\newcommand{\cV}{\mathcal V}
\newcommand{\cH}{\mathcal H}
\newcommand{\cD}{\mathcal D}
\newcommand{\cW}{\mathcal W}
\newcommand{\bR}{\mathbb R}

\DeclareMathOperator{\meas}{meas}

\journal{CMAME}

\begin{document}

\begin{frontmatter}

\title{Discontinuous Galerkin method for an integro-differential equation 
modeling dynamic fractional order viscoelasticity}

\author{Stig Larsson}
\ead{stig@chalmers.se}

\address{Department of Mathematical Sciences,
  Chalmers University of Technology and University of Gothenburg,
  SE--412 96 Gothenburg,
  Sweden
}

\author{Milena Racheva}
\ead{milena@tugab.bg}

\address{Department of Mathematics,
  Technical University of Gabrovo, 5300 Gabrovo, Bulgaria
}

\author{Fardin Saedpanah}
\ead{f.saedpanah@uok.ac.ir; f\_saedpanah@yahoo.com}

\address{
Department of Mathematics, 
University of Kurdistan, P. O. Box 416, 
Sanandaj, Iran
}

%



\begin{abstract}
An integro-differential equation, modeling dynamic fractional order 
viscoelasticity, with a Mittag-Leffler type convolution kernel is considered. 
A discontinuous Galerkin method, based on piecewise constant 
polynomials is formulated for temporal semidiscretization of the problem. 
Stability estimates of the discrete problem are proved, that are used 
to prove optimal order a priori error estimates. The theory is illustrated by 
a numerical example. 
\end{abstract}

\begin{keyword}
integro-differential equation, 
fractional order viscoelasticity, discontinuous Galerkin method, 
weakly singular kernel, stability, a priori estimate.
\end{keyword}



\end{frontmatter}

\section{Introduction}
Fractional order integral/differential operators have proved to be very suitable 
for modeling memory effects of various materials, \cite{BagleyTorvik1983}. 
In particular, for modeling viscoelastic materials, for more details and references 
see \cite{FardinEJM2014}.  
The basic equations of the viscoelastic dynamic problem, that is a 
hyperbolic type integro-differential equations, 
can be written in the strong form,
\begin{equation}\label{strongform}
  \begin{aligned}
    &\rho\ddot{u}(x,t)-\nabla\cdot\sigma_0
    (u;x,t)
    \\
    &\quad+\int_0^t \!\beta(t-s)\nabla\cdot \sigma_0(u;x,s)\,
    ds=f(x,t)\quad&& \textrm{in} \;\,\Omega\times (0,T),\\
    &u(x,t)=0 \quad&& \textrm{on}\;\Gamma_{D}\times (0,T),\\
    &\sigma(u;x,t)\cdot n(x)=g(x,t)
       \quad&&\textrm{on}\;\Gamma_{N}\times (0,T),\\
    &u(x,0)=u_0(x)\quad&&\textrm{in}\;\,\Omega,\\
    &\dot{u}(x,0)=v_0(x)\quad&&\textrm{in}\;\,\Omega,
  \end{aligned}
\end{equation}
(throughout this text we use `$\cdot$' to denote `$\frac{\partial}{\partial t}$')
where $u$ is the displacement vector, $\rho$ is the (constant) mass density,
$f$ and $g$ represent, respectively, the volume and surface loads, 
$\sigma_0$ is an elastic stress according to
\begin{equation*} 
    \sigma_0(u)=2\mu \epsilon(u)+\lambda{\rm{Tr}}(\epsilon(u))
      {\rm  I},
\end{equation*}
and 
the stress is
\begin{equation*} 
  \sigma(u)=\sigma_0(u)-\int_0^t \beta(t-s)\sigma_0(u(s))\,ds,
\end{equation*}
where $\lambda , \mu>0$ are elastic constants of
Lam$\acute{\rm e}$ type, $\epsilon$
is the strain which is defined by 
$\epsilon(u)=\frac{1}{2}\big(\nabla u+(\nabla u)^T\big)$. 
Here, $\beta$ 
is the convolution kernel
\begin{equation} \label{beta}
  \begin{split}
    \beta(t) &= - \gamma\frac{d}{dt}
    \Big({\rm E}_{\alpha}(-(t/\tau)^{\alpha})\Big)
     = \gamma\frac{\alpha}{\tau}\Big(\frac{t}{\tau}\Big)^{\alpha-1}
        {\rm E}_\alpha'\Big(-\big(\frac{t}{\tau}\big)^{\alpha}\Big)
    \approx Ct^{-1+\alpha},\,\,t \to 0,
  \end{split}
\end{equation}
where $0<\gamma<1$, $\tau>0$ is the relaxation time and
 ${\rm E}_{\alpha}(z) = \sum_{k=0}^{\infty}\frac{z^k}{\Gamma(1+\alpha k)}$
is the Mittag-Leffler function of order $\alpha\in(0,1)$. 
The convolution kernel is weakly singular and $\beta\in L_1(0,\infty)$ with
$\int_0^\infty \!\beta(t)\,dt=\gamma$. 

Well-posedness of the model problem \eqref{strongform} and more general form 
of such equations in fractional order viscoelasticity have been studied in 
\cite{FardinEJM2014}, by means of Galerkin approximation methods. 
Continuous Galerkin methods of order one, both in time 
and space variables, have been applied to similar problems in 
\cite{StigFardin2010}, \cite{FardinBIT2013} and \cite{FardinIMAJNUMANAL2014}. 
Discontinuous Galerkin and continuous Galerkin method, respectively, 
in time and space variables have been applied to a dynamic model problem 
in linear viscoelasticity (with exponential kernels) in \cite{RiviereShawWhiteman2007}. 
For more references on numerical and analytical treatment of 
integro-differential equations, among the extensive literature, see e.g., 
\cite{LinThomeeWahlbin1991}, \cite{PaniThomeeWahlbin1992}, 
\cite{McLeanSloanThomee2006}, \cite{ShawWhiteman2004}, 
\cite{AdolfssonEnelundLarssonRacheva2006}, and their references. 

Here, we formulate the discontinuous Galerkin method dG(0), 
based on piecewise constant polynomials in the time variable, 
for the temporal semidiscretization of the problem. 
We prove stability estimates for the 
discrete problem, that are used to prove optimal order a priori error estimates 
for the displacement $u$ and velocity $\dot u$. 
Then we illustrate the theory 
by a numerical example. 
The present work extends previous works, e.g., 
\cite{ShawWhiteman2004} and \cite{AdolfssonEnelundLarsson2008} 
on quasi-static $(\rho \ddot{u}\approx 0)$ linear and fractional order 
viscoelasticity,  to the dynamic fractional order case. 

The convolution integral in the model problem generates a growing amount of 
data that has to be stored and used in each time step. 
Lubich's convolution quadrature \cite{Lubich1988}, that has been  
improved in \cite{SchadleLopezLubich2008},  has been commonly 
used for this integration. See \cite{AdolfssonEnelundLarsson2008} and 
references therein for examples of application of this approach and a different 
approach, the so-called ``sparse quadrature'', 
that was introduced  in \cite{SloanThomee1986}, but only for the case of a kernel 
without singularity. See \cite{AdolfssonEnelundLarsson2003}, where 
the same procedure has been extended to the case of the singular kernel.
We note that, when the exponential decaying kernel, 
in linear viscoelasticity, is represented as a Prony series, 
it results in a recurrence formula 
for history updating, see \cite{ShawWhiteman2004}. 
This means that, in this case we do not use convolution 
quadrature. 

In general we do not have global regularity of solutions, 
see \cite{FardinEJM2014}, due to regularity of the kernel and 
mixed boundary conditions,  
which calls for adaptive methods based on a posteriori error analysis. 
We plan to address these issues (numerical adaptation methods together 
with sparse quadrature) and full discrete space-time 
discontinuous Galerkin and continuous Galerkin methods  
in future work.

In the next section, we provide some definitions and the weak formulations 
of the model problem. In $\S 3$ we formulate the discontinuous Galerkin 
method. Then in $\S 4$ we show an energy identity and stability estimates 
for the discrete problem, that is used in $\S 5$ to prove optimal order 
a priori error estimates. Finally, in $\S6$, we illustrate that the dG(0) method 
capture the mechanical behavior of the model problem and we  
investigate the rate of convergence $O(k)$, by a numerical example. 

\section{Preliminaries}
We let $\Omega \subset\mathbb{R}^d, \, d=2,3$, be a bounded polygonal domain 
with boundary $\Gamma=\Gamma_D\cup\Gamma_N$, where $\Gamma_D$ and $\Gamma_N$ are 
disjoint and $\meas(\Gamma_D)\not=0$.
We introduce the function spaces
$\cH=L_2(\Omega)^d,\,\cH_{\Gamma_N}=L_2(\Gamma_N)^d,\,$ 
 and $\cV=\{v\in H^1(\Omega)^d:v\!\!\mid_{\Gamma_D}=0\}$.
We denote the norms in $\cH$ and $\cH_{\Gamma_N}$ by
$\lVert\cdot\rVert$ and $\lVert\cdot\rVert_{\Gamma_N}$, respectively, 
and we equip $\cV$ with the inner product
$a(\cdot,\cdot)$ and norm $\lVert v\rVert_{\cV}^2=a(v,v)$,
where (with the usual summation convention)
\begin{equation}\label{bilinear}
 a(v,w)=\int_{\Omega}\!\big(2\mu\epsilon_{ij}(v)\epsilon_{ij}(w) +
\lambda\epsilon_{ii}(v)\epsilon_{jj}(w)\big)\,dx,\quad v,w\in \cV\,,
\end{equation}
which is a coercive bilinear form on $\cV$.

Now, we can write the weak form of the equation of motion as:
Find $u(t)\in \cV$ such that
$u(0)=u_0$, $\dot{u}(0)=v_0,$ and
\begin{equation}\label{weakform1}
  \begin{split}
    \rho (\ddot{u}(t),v) + a(u(t),v)
    -& \int_0^t \beta(t-s) a(u(s),v) \,ds \\
     &= (f(t),v) + (g(t),v)_{\Gamma_N},
    \quad \forall v \in \cV ,\,t\in (0,T),
  \end{split}
\end{equation}
with $(g(t),v)_{\Gamma_N}=\int_{\Gamma_N}\! g(t)\cdot v\,dS$.

Defining the new variables $u_1=u$ and $u_2=\dot{u}$
we write the velocity-displacement form of \eqref{weakform1} as: 
Find $u_1(t),\,u_2(t)\in \cV$ such that
$u_1(0)=u_0$, $u_2(0)=v_0$, and
\begin{equation}\label{weakform2}
  \begin{split}
    &a\big(\dot{u}_1(t),v_1\big)-a\big(u_2(t),v_1)=0,\\
    &\rho (\dot{u}_2(t),v_2) + a(u_1(t),v_2)
    - \int_0^t \beta(t-s) a(u_1(s), v_2) \, ds \\
    &\qquad\qquad\qquad\quad\quad
           = (f(t),v_2) + (g(t),v_2)_{\Gamma_N},
    \quad \forall v_1,v_2 \in \cV ,\,t\in (0,T),
  \end{split}
\end{equation}
that is used for discontinuous Galerkin formulation. 

We recall that the positive convolution kernel $\beta$ is weakly singular, 
that is, $\beta>0$ is singular at the origin, but 
$\| \beta \|_{L_1(0,\infty)}=\gamma<1$. 
For our analysis, we define the function 
\begin{equation} \label{eta}
  \eta (t)=1-\int_0^t \beta(s)\, ds,
\end{equation}
and it is easy to see that 
\begin{equation} \label{eta-property}
  \eta(0)=1,\quad \lim_{t\to\infty} \eta(t)=1-\gamma <1,\quad \dot \eta(t)=-\beta(t).
\end{equation}

We set $Au=-\nabla\cdot \sigma_0(u)$
with $\mathcal{D}(A)=H^2(\Omega)^d\cap \cV$, 
and we note that with homogeneous boundary conditions 
in \eqref{strongform}, i.e. $\Gamma_D=\Gamma$ or $g=0$, 
we have  $a(u,v)=(Au,v)$ for $u \in \cD(A)$, $v \in \cV$. 
It is known that $A$ can be extended to 
a self-adjoint, positive definite, unbounded operator on $\cH$. 
Then we may define $A^l,\ l \in \bR$ by 
\begin{equation*}
  A^l v =\sum_{k=1}^\infty \lambda_k^l (v ,\varphi_k) \varphi_k,
\end{equation*}
where $\{(\lambda_k, \varphi_k)\}_{k=1}^\infty$ are the eigenpairs of 
the operator $A$, see e.g., \cite{Thomee_Book}. 
We also use the norms
\begin{equation*}
  \|v\|_l=\|A^{l/2} v\|=\sqrt{(A^l v,v)}, \qquad l \in \bR,
\end{equation*}
and denote $\|v\|=\|v\|_0=\|v\|_\cH$. 
See \cite{FardinEJM2014} for more details and the regularity of the 
solution of the model problem \eqref{strongform}.

\section{The discontinuous Galerkin method}
Here we formulate the discontinuous Galerkin method, dG(0), that is 
based on piecewise constant polynomials, for temporal discretization 
of the model problem \eqref{strongform} with the weak 
form \eqref{weakform2}. 

Let $0 = t_0 < t_1 , \ldots < t_N = T$ be a temporal mesh, $I_n = (t_{n-1},t_n)$ 
denote the time intervals and $k_n = t_n - t_{n-1}$ denote the time steps. 
The discrete finite element  space is 
\begin{equation*}
  \cW_D = \left\{ w = (w_1,w_2): 
  w_i|_{I_n} = w_{i,n} \in \cD(A), \ n = 1,\ldots,N\right\}.
\end{equation*}
We note that $w \in \cW_D$ is piecewise constant in time and in general 
is not continuous at the time nodes $t_n, n=1,\ldots,N$, so we use the following notations: 
$w_n = w|_{I_n} = w_{n-1}^+ = w_n^-$ 
and $[w]_n = w_n^+-w_n^-=w_{n+1} - w_n$ for the jump terms.

Then, recalling \eqref{weakform2},  the dG(0) method is to find 
$U=(U_1,U_2) \in \cW_D$ such that
\begin{equation} \label{dG0Expand}
  \begin{split}
    &\int_{I_n}\!\Big( a(\dot{U}_1,V_1)-a(U_2,V_1)\Big)\ dt
       +a([U_1]_{n-1},V_{1,n-1}^+)=0,\\
    &\int_{I_n}\!\Big(\rho (\dot{U}_2,V_2) + a(U_1,V_2)
    - \int_0^t \beta(t-s) a(U_1(s), V_2) \ ds \Big)\ dt \\
    &\qquad\qquad +\rho([U_2]_{n-1},V_{2,n-1}^+) \\
    &\qquad \qquad\qquad\qquad = 
    \int_{I_n}\! \Big((f,V_2) + (g,V_2)_{\Gamma_N} \Big)\ dt,
    \quad \forall V=(V_1,V_2) \in \cW_D ,\,t\in (0,T),\\
    &U_{1,0}^- =u_0, \quad U_{2,0}^-=v_0. 
  \end{split}
\end{equation}
Introducing an abstract operator $\tilde A: \cV \to \cV^*$, that is equivalent to $A$ 
with homogeneous Neumann boundary condition, and recalling the fact that, 
the functions in $\cW_D$ are piecewise constant with respect to time, we get, 
with $U_{1,0}=u_0$, $U_{2,0}=v_0$,
\begin{equation*} 
  \begin{split}
    &\tilde AU_{1,n}-k_n \tilde AU_{2,n}=\tilde AU_{1,n-1},\\
    &(k_n-\omega_{nn}) \tilde AU_{1,n}+\rho U_{2,n}
        =\rho U_{2,n-1} +\sum_{j=1}^{n-1} \omega_{nj} \tilde A U_{1,j} 
          +k_n ( \bar f_n +\bar g_n),
  \end{split}
\end{equation*}
where obviously for $n=1$ the sum on the right side is ignored  and 
\begin{equation}   \label{omega_fn_gn}
  \begin{split}
   \omega_{nj}&=\int_{I_n}\int_{t_{j-1}}^{t_j \wedge t }\beta(t-s)\ ds\ dt, \quad
   t_j\wedge t = \min (t_j,t),\\
   \bar f_n&=\frac{1}{k_n} \int_{I_n}\! f(t)\ dt, \quad
   \bar g_n=\frac{1}{k_n} \int_{I_n}\! \tilde g(t)\ dt ,
 \end{split}
\end{equation}
with $\tilde g(t) \in \cV^*$ such that 
$\langle \tilde g(t),v \rangle = (g(t),v)_{\Gamma_N}$, $\forall v \in \cV,\ t \in [0,T]$. 
This is used for computer implementations. 

Now, we define the function space $\cW$ that consists of functions 
that are piecewise smooth with respect to the temporal mesh with 
values in $\cD (A)$. We note that $\cW_D \subset \cW$. 
Then we define the bilinear form $B:\cW \times \cW \to \bR$ and the linear form 
$L:\cW \to \bR$ by 
\begin{equation*}
  \begin{split}
    B((u_1,u_2),(v_1,v_2))
    &=\sum_{n=1}^N \int_{I_n} \Big \{
    a(\dot u_1,v_1)-a(u_2,v_1)\\
    &\quad +\rho(\dot u_2,v_2)+a(u_1,v_2)
    -\int_0^t \beta(t-s)a(u_1(s),v_2(t))\, ds \Big \}\, dt\\
    &\quad +\sum_{n=1}^{N-1}\left \{a([u_1]_n,v_{1,n}^+) 
                                                          +\rho([u_2]_n,v_{2,n}^+) \right \}\\
    &\quad + a(u_{1,0}^+,v_{1,0}^+) + \rho(u_{2,0}^+,v_{2,0}^+),\\
    L((v_1,v_2))
    &=\sum_{n=1}^N \int_{I_n} \Big( (f,v_2) 
    + (g,v_2)_{\Gamma_N} \Big) \  dt
         + a(u_0,v_{1,0}^+) + \rho(v_0,v_{2,0}^+).
  \end{split}
\end{equation*}
Then $U=(U_1,U_2) \in \cW_D$, the solution of the discrete problem 
\eqref{dG0Expand}, satisfies
\begin{equation} \label{dG0}
  \begin{split}
    &B(U,V)=L(V),\qquad \forall V=(V_1,V_2) \in \cW_D, \\
    &U_0^-=(U_{1,0}^-,U_{2,0}^-)=(u_0,v_0).
  \end{split}
\end{equation}

We note that the solution $(u_1,u_2)$ of \eqref{weakform2} also 
satisfies 
\begin{equation} \label{weakform2Compact}
  \begin{split}
    &B((u_1,u_2),(v_1,v_2))=L((v_1,v_2)),\qquad \forall (v_1,v_2) \in \cW, \\
    &(u_1(0),u_2(0))=(u_0,v_0),
  \end{split}
\end{equation}
such that the Galerkin's orthogonality holds for the error 
$e=(e_1,e_2)=(U_1,U_2)-(u_1,u_2)$, that is, 
\begin{equation} \label{GalerkinOrthogonality}
   B(e,V)=0,\qquad \forall V=(V_1,V_2) \in \cW_D.
\end{equation}

\section{Stability}
Here we prove a stability identity and stability estimates that are used in 
a priori error analysis. To this end, we need to prove a stability identity 
for a slightly different problem, that is, 
$U \in \cW_D$ such that
\begin{equation} \label{dG0-General}
  \begin{split}
    &B(U,V)=\hat L(V),\qquad \forall V \in \cW_D, \\
    &U_0^-=(U_{1,0}^-,U_{2,0}^-)=(u_0,v_0),
  \end{split}
\end{equation}
where the linear form $\hat L: \cW \to \bR$ is defined by 
\begin{equation*}
  \hat L((v_1,v_2))=\sum_{n=1}^N \int_{I_n} a(f_1,v_1)+ (f_2,v_2)\, dt
         + a(u_0,v_{1,0}^+) + \rho(v_0,v_{2,0}^+).
\end{equation*}

These terms are dictated by the error equation in \eqref{ThetaEquation} below. 
Note in particular that no traction data term $(g,v_2)_{\Gamma_N}$ is needed. 
Recalling $\eta$ from \eqref{eta}, we define 
\begin{equation} \label{eta-n}
  \eta_n=\frac{1}{k_n} \int_{I_n} \eta (t)\, dt 
  =1-\frac{1}{k_n}\int_{I_n}\int_0^t \beta(s)\, ds\,dt,
\end{equation}
with $\eta_0=1$. 
We also denote the backward difference operator, for $V_n$,
\begin{equation} \label{BackwardDifference}
  \partial_n V_n= \frac{V_n-V_{n-1}}{k_n}.
\end{equation}
Obviously we have
\begin{equation} \label{BackwardDifference-p1}
  \begin{split}
    k_n \partial_n (W_nV_n)
    &=W_nV_n - W_{n-1}V_{n-1}\\
    &=W_nV_n - W_{n-1}V_n + W_{n-1}V_n - W_{n-1}V_{n-1}\\
    &=k_n \partial_n W_n V_n + k_n W_{n-1} \partial_n V_n,
  \end{split}
\end{equation}
that also implies
\begin{equation} \label{BackwardDifference-p2}
  \begin{split}
    \partial_n (V_nV_n) + k_n (\partial_n V_n \partial_n V_n)
    &=V_n \partial_n V_n+V_{n-1}\partial_n V_n+\partial_n V_n k_n \partial_n V_n\\
    &=\partial_n V_n(V_n+V_{n-1}+V_n-V_{n-1})\\
    &=2V_n \partial_n V_n.
  \end{split}
\end{equation}

We also define the standard $L_2$-projection 
$P_{k,n}:L_2(I_n)^d \to \mathbb{P}_0^d(I_n)$ by 
\begin{equation*}  
  \int_{I_n} \! (P_{k,n} v -v) \ dt=0, \quad \forall v\in L_2(I_n)^d,  
\end{equation*}
where $\mathbb{P}_0^d$ denotes all vector valued constant polynomials on $I_n$. 
We use the obvious notation $P_k$ over the interval $(0,T)$, i.e, 
$P_{k,n} = P_k|_{I_n}$. 
It is easy to see that 
\begin{equation}  \label{PknProperty}
  P_{k,n} v = \bar v = \frac{1}{k_n} \int_{I_n} \! v \ dt, \quad \forall v\in L_2(I_n)^d.   
\end{equation}

\begin{theorem}
Let $U=(U_1,U_2)$ be a solution of \eqref{dG0-General}. 
Then for any $T>0$ and 
$l\in \{0,-1\}$, 
we have the equality
\begin{equation} \label{dG0-StabilityIdentity}
  \begin{split}
    \eta_N  \| U_{1,N}\|_{l+1}^2&+\rho \| U_{2,N}\|_{l}^2\\
    &+\sum_{n=1}^N k_n 
      \Big\{-\partial_n \eta_n \| U_{1,n-1}\|_{l+1}^2
        +k_n\eta_n \|\partial_n U_{1,n}\|_{l+1}^2\Big\}\\
    &+\sum_{n=2}^N\sum_{j=1}^{n-1}
      \int_{I_n}\int_{I_j} \beta(t-s)\, ds\, dt
      \Big\{ \partial_n\|W_{1,n,j}\|_{l+1}^2 + k_n \|\partial_n W_{1,n,j}\|_{l+1}^2 \Big\}\\
    &+\rho \sum_{n=0}^{N-1} \| [U_2]_n\|_{l}^2\\
    &\!\!\!\!\! = \|u_0\|_{l+1}^2 +\rho \|v_0\|_{l}^2 \\
    &+2\int_0^T \big \{ \eta  a(P_k f_1,A^l U_1) + (P_k f_2,A^l U_2)\big\} \, dt\\
    &+2\int_0^T\int_0^t  \beta(t-s) a\Big(P_k f_1(t),A^l(U_1(t)-U_1(s))\Big) \, ds \, dt,
  \end{split}
\end{equation}
where $W_{1,n,j}=U_{1,n} - U_{1,j}$. 
All terms on the left side are non-negative. 

Moreover, for some $C=C(\gamma,\rho)$, we have the stability estimate
\begin{equation} \label{dG0-StabilityEstimate}
  \| U_{1,N}\|_{l+1}+\| U_{2,N}\|_{l}
  \leq  C \Big\{
  \|u_0\|_{l+1} + \|v_0\|_{l} +\int_0^T \|f_1\|_{l+1} + \|f_2\|_{l} \, dt \Big\}.
\end{equation}
\end{theorem}
\begin{proof}
We organize our proof in five steps. 

1. First, we find a representation of $U_2$ in terms of $U_1$ and $f_1$. 
Setting $V_2=0$ in \eqref{dG0-General}, we have
\begin{equation*}
  \begin{split}
    \sum_{n=1}^N \int_{I_n} \big \{
      a(\dot U_1,V_1)&-a(U_2,V_1) \big \} \, dt
      +\sum_{n=1}^{N-1}a([U_1]_n,V_{1,n}^+) + a(U_{1,0}^+,V_{1,0}^+)\\ 
    &=\sum_{n=1}^N \int_{I_n} a(f_1,V_1)\, dt + a(u_0,V_{1,0}^+),
  \end{split}
\end{equation*}
that, considering the fact that $U_i,\  i=1,2$, are piecewise constant with respect to 
time, $\dot U_1=0$ and recalling \eqref{PknProperty}, we have
\begin{equation*}
  \begin{split}
    -\sum_{n=1}^N k_n &a(U_{2,n},V_{1,n}) 
    +\sum_{n=2}^{N} a([U_1]_{n-1},V_{1,n}) + a(U_{1,1},V_{1,1}) \\
    &=\sum_{n=1}^N k_n a(P_{k,n}f_1,V_{1,n})\, dt + a(u_0,V_{1,1}).
  \end{split}
\end{equation*}
Now, for some $n \in \{1,\dots,N\}$, we take $V_{1,n} \neq 0$ and $V_1=0$ otherwise, 
and we have
\begin{equation*}
  -k_n a(U_{2,n},V_{1,n}) + a(U_{1,n}-U_{1,n-1},V_{1,n})=k_n a(P_{k,n}f_1,V_{1,n}),
\end{equation*}
that implies
\begin{equation} \label{U2-U1}
  U_{2,n}=\frac{U_{1,n}-U_{1,n-1}}{k_n} - P_{k,n}f_1
  =\partial_nU_{1,n} - P_{k,n}f_1.
\end{equation}

2. Now, recalling function $\eta$ from \eqref{eta}, we use the representation 
\begin{equation*}
  \begin{split}
    a(U_1,V_2)&-\int_0^t \beta(t-s) a(U_1(s),V_2(t))\, ds\\
    &=\eta(t) a(U_1,V_2) + \int_0^t \beta(t-s) a(U_1(t)-U_1(s),V_2(t))\, ds,
  \end{split}
\end{equation*}
and we set $V=A^l U,\ l \in \{0,-1\}$ in \eqref{dG0-General}, to obtain
\begin{equation} \label{Stability-eq1}
  \begin{split}
    \sum_{n=1}^N \int_{I_n} 
    &\Big \{
      a(\dot U_1,A^l U_1)-a(U_2,A^l U_1)
      +\rho(\dot U_2,A^l U_2)+\eta(t)a(U_1,A^l U_2)\\
    &\quad +\int_0^t \beta(t-s)a(U_1(t)-U_1(s),A^l U_2(t))\, ds \Big \}\, dt\\
    &\quad +\sum_{n=1}^{N-1}\left \{a([U_1]_n,A^l U_{1,n}^+) 
                                                          +\rho ([U_2]_n,A^l U_{2,n}^+) \right \}\\
    &\quad + a(U_{1,0}^+,A^l U_{1,0}^+) + \rho(U_{2,0}^+,A^l U_{2,0}^+)\\
    &=\int_0^T a(f_1,A^l U_1)+(f_2,A^l U_2)\, dt 
      +a(u_0,A^l U_{1,0}^+) + \rho (v_0,A^l U_{2,0}^+).
  \end{split}
\end{equation}
Then, using \eqref{U2-U1} and $\dot U_1=0$ we have
\begin{equation*}
  \begin{split}
    \sum_{n=1}^N &\int_{I_n} \big \{
      a(\dot U_1,A^l U_1)-a(U_2,A^l U_1)\big \}\, dt
    +\sum_{n=1}^{N-1}a([U_1]_n,A^l U_{1,n}^+)+a(U_{1,0}^+,A^l U_{1,0}^+)\\
    &=\sum_{n=1}^N \int_{I_n} -a(\partial_n U_{1,n}-P_{k,n}f_1,A^l U_{1,n})\, dt\\
    &\quad +\sum_{n=1}^{N-1} a(U_{1,n+1}-U_{1,n},A^l U_{1,n+1})
                 +a(U_{1,1},A^l U_{1,1}) \\
    &=-\sum_{n=1}^N k_n a\Big(\frac{U_{1,n}-U_{1,n-1}}{k_n}-P_{k,n}f_1,A^l U_{1,n}\Big)\, dt\\
    &\quad +\sum_{n=1}^{N-1} a(U_{1,n+1}-U_{1,n},A^l U_{1,n+1})
                 +a(U_{1,1},A^l U_{1,1}) ,
  \end{split}
\end{equation*}
that, recalling \eqref{PknProperty}, we have 
\begin{equation*}
  \begin{split}
    \sum_{n=1}^N &\int_{I_n} \big \{
      a(\dot U_1,A^l U_1)-a(U_2,A^l U_1)\big \}\, dt
    +\sum_{n=1}^{N-1}a([U_1]_n,A^l U_{1,n}^+)+a(U_{1,0}^+,A^l U_{1,0}^+)\\
    &=-\sum_{n=2}^{N} a(U_{1,n},A^l U_{1,n})+\sum_{n=2}^{N} a(U_{1,n-1},A^l U_{1,n})\\
    &\qquad  -a(U_{1,1},A^l U_{1,1})+a(U_{1,0},A^l U_{1,1})\\
    &\quad +\sum_{n=2}^{N} a(U_{1,n},A^l U_{1,n})
                 -\sum_{n=2}^{N} a(U_{1,n-1},A^l U_{1,n})\\
     &\qquad +a(U_{1,1},A^l U_{1,1})+\sum_{n=1}^N\int_{I_n} a(f_1,A^l U_{1,n})\, dt\\
    &=a(U_{1,0},A^l U_{1,1})+\int_0^T a(f_1,A^l U_1)\, dt\\
    &=a(u_0,A^l U_{1,0}^+)+\int_0^T a(f_1,A^l U_1)\, dt .
  \end{split}
\end{equation*}
From this, $\dot U_2=0$ 
and the definition of the $L_2$ projection $P_k$, 
we can write \eqref{Stability-eq1} as 
\begin{equation} \label{Stability-eq2}
  \begin{split}
    &\sum_{n=1}^N \int_{I_n}  \eta(t)a(U_1,A^l U_2) \, dt\\
    &\quad +\sum_{n=1}^N \int_{I_n} \int_0^t 
                  \beta(t-s)a(U_1(t)-U_1(s),A^l U_2(t))\, ds \, dt\\
    &\quad\quad +\rho\sum_{n=1}^{N-1} ([U_2]_n,A^l U_{2,n}^+) 
       +\rho(U_{2,0}^+,A^l U_{2,0}^+)
       -\rho(v_0,A^l U_{2,0}^+)\\
    &=\int_0^T (f_2,A^l U_2)\, dt \\
    &= \int_0^T (P_k f_2,A^l U_2)\, dt.
  \end{split}
\end{equation}
Now, we need to study the three terms on the left side. 

3. For the first term on the left side of \eqref{Stability-eq2}, 
recalling \eqref{U2-U1} and $\eta_n$ from \eqref{eta-n}, 
we have
\begin{equation*}
  \begin{split}
    \sum_{n=1}^N \int_{I_n} \eta(t) a(U_1(t),A^l U_2(t)) \,dt
    &=\sum_{n=1}^N \int_{I_n} \eta(t)   a(U_{1,n},A^l U_{2,n})\ dt\\
    &= \sum_{n=1}^N k_n \eta_{n} a(A^{l/2} U_{1,n},\partial_n A^{l/2}  U_{1,n})\\
    &\quad -\sum_{n=1}^N \int_{I_n} \eta(t) a(U_{1,n},A^l P_{k,n}f_1(t))\ dt,
  \end{split}
\end{equation*}
that, using \eqref{BackwardDifference-p1} and \eqref{BackwardDifference-p2}, 
implies 
\begin{equation*}
  \begin{split}
    \sum_{n=1}^N \int_{I_n} &\eta(t) a(U_1(t),A^l U_2(t)) \,dt \\
    &=\frac{1}{2} \sum_{n=1}^N k_n \eta_{n} 
        \left\{\partial_n a(A^{l/2} U_{1,n},A^{l/2} U_{1,n}) 
        + k_n a(\partial_n A^{l/2} U_{1,n},\partial_n A^{l/2} U_{1,n})\right\}\\
    &\quad -\sum_{n=1}^N \int_{I_n}\eta a(U_{1,n},A^l P_{k,n}f_1)\ dt\\
    &=\frac{1}{2} \sum_{n=1}^N  k_n 
          \partial_n\left\{\eta_{n} a(A^{l/2} U_{1,n},A^{l/2} U_{1,n})\right\}\\
    &\quad - \frac{1}{2} \sum_{n=1}^N k_n 
          a(A^{l/2} U_{1,n-1},A^{l/2} U_{1,n-1}) \partial_n \eta_{n}\\
    &\quad + \frac{1}{2} \sum_{n=1}^N  
        k_n^2 \eta_{n} a(\partial_n A^{l/2} U_{1,n},\partial_n A^{l/2} U_{1,n})\\
    &\quad -\sum_{n=1}^N \int_{I_n} \eta a(U_{1,n},A^l P_{k,n}f_1)\,dt, 
  \end{split}
\end{equation*}
so we have 
\begin{equation*}
  \begin{split}
    \sum_{n=1}^N \int_{I_n} &\eta(t) a(U_1(t),A^l U_2(t)) \,dt \\
    &= \frac{1}{2} \eta_{N} a(A^{l/2} U_{1,N},A^{l/2} U_{1,N}) 
        - \frac{1}{2} \eta_{0} a(A^{l/2} U_{1,0},A^{l/2} U_{1,0})\\
    &\quad - \frac{1}{2} \sum_{n=1}^N k_n 
          a(A^{l/2} U_{1,n-1},A^{l/2} U_{1,n-1}) \partial_n \eta_{n} \\
    &\quad + \frac{1}{2} \sum_{n=1}^N 
        k_n^2 \eta_{n} a(\partial_n A^{l/2} U_{1,n},\partial_n A^{l/2} U_{1,n})
        -\int_0^T \eta a(U_1,A^l P_k f_1)\, dt. 
  \end{split}
\end{equation*}
Consequently, we have
\begin{equation} \label{Stability-term1}
  \begin{split}
    \sum_{n=1}^N \int_{I_n} &\eta(t) a(U_1(t),A^l U_2(t)) \,dt \\
    &=\frac{1}{2} \eta_{N} \|U_{1,N}\|_{l+1}^2 
        - \frac{1}{2} \| u_0 \|_{l+1}^2\\
    &\quad + \frac{1}{2} \sum_{n=1}^N k_n \big\{-\partial_n \eta_{n}\|U_{1,n-1}\|_{l+1}^2 
        + k_n \eta_{n} \| \partial_n U_{1,n}\|_{l+1}^2 \big\}\\
    &\quad -\int_0^T \eta a(U_1,A^l P_k f_1)\, dt.
  \end{split}
\end{equation}
Here, we note that $\partial_n \eta_n<0$. 
Indeed, changing the variable 
$t= t_{n-1} + k_n s$, for $t \in I_n$, $n \ge 2$, we have
\begin{equation*}
  \eta_{n} = \frac{1}{k_n} \int_{I_n} \eta(t) \,dt 
  = \int_0^1 \eta(t_{n-1}+sk_n)\,ds,
\end{equation*}
that implies 
\begin{equation*}
  \partial_n \eta_{n} 
  = \frac{1}{k_n} \int_0^1 \Big(\eta(t_{n-1}+sk_n)- \eta(t_{n-2}+sk_{n-1})\Big)\,ds < 0,
\end{equation*}
since $\eta$ is a decreasing function by \eqref{eta-property}.
And, for $n=1$, we have
\begin{equation*}
  \partial_1\eta_{1} 
  = \frac{1}{k_1} (\eta_{1} - \eta_{0}) 
  = -\frac{1}{k_1^2}\int_{I_1}\int_0^t \beta(s)\,ds\,dt < 0.
\end{equation*}

Now, we study the second term on the left side of \eqref{Stability-eq2}, 
that is the convolution integral. Recalling \eqref{U2-U1} 
and noting that $U_{1,n}-U_{1,j}=0$ for $n=j$, 
we have
\begin{equation*}
  \begin{split}
    \sum_{n=1}^N \int_{I_n} &\int_0^t \beta(t-s) a(U_1(t) -U_1(s),A^l U_2(t))\,ds\,dt\\
    &= \sum_{n=1}^N \sum_{j=1}^n 
         \int_{I_n}\int_{t_{j-1}}^{t_j\wedge t}
            \beta(t-s)\,ds\,dt \, 
              a\big(U_{1,n} - U_{1,j},A^l (\partial_n U_{1,n} -P_{k,n}f_1)\big)\\
    &= \sum_{n=2}^N \sum_{j=1}^{n-1}\int_{I_n}\int_{I_j}
             \beta(t-s)\,ds\,dt \, 
             a\big(U_{1,n} - U_{1,j},A^l (\partial_n U_{1,n}-P_{k,n}f_1)\big). 
  \end{split}
\end{equation*}
Then, recalling $W_{1,n,j}=U_{1,n}-U_{1,j}$ and using 
\begin{equation*}
\partial_n U_{1,n}=\frac{U_{1,n}-U_{1,n-1}}{k_n}
=\frac{(U_{1,n}-U_{1,j})-(U_{1,n-1}-U_{1,j})}{k_n}
=\frac{W_{1,n,j}-W_{1,n-1,j}}{k_n}
=\partial_n W_{1,n,j},
\end{equation*}
we have
\begin{equation*}
  \begin{split}
    \sum_{n=1}^N \int_{I_n} &\int_0^t \beta(t-s) a(U_1(t) -U_1(s),A^l U_2(t))\,ds\,dt\\
    &=\sum_{n=2}^N \sum_{j=1}^{n-1}\int_{I_n}\int_{I_j}
         \beta(t-s)\,ds\,dt\, a\big(W_{1,n,j},A^l (\partial_n W_{1,n,j} -P_{k,n}f_1)\big) .
  \end{split}
\end{equation*}
Using also 
\eqref{BackwardDifference-p2}, this yields
\begin{equation*}
  \begin{split}
    \sum_{n=1}^N \int_{I_n} &\int_0^t \beta(t-s) a(U_1(t) -U_1(s),A^l U_2(t))\,ds\,dt\\
    &= \frac{1}{2}\sum_{n=2}^N \sum_{j=1}^{n-1}\int_{I_n}\int_{I_j}
          \beta(t-s)\,ds\,dt\, \partial_n a(A^{l/2} W_{1,n,j},A^{l/2} W_{1,n,j})\\
    &\quad + \frac{1}{2}\sum_{n=2}^N 
                    k_n \sum_{j=1}^{n-1}\int_{I_n}\int_{I_j}
                    \beta(t-s)\,ds\,dt\, a(\partial_n A^{l/2} W_{1,n,j},\partial_n A^{l/2} W_{1,n,j})\\
    &\quad -\sum_{n=2}^N \sum_{j=1}^{n-1}\int_{I_n}\int_{I_j}
         \beta(t-s)\,ds\,dt\, a(W_{1,n,j},A^l P_{k,n}f_1).
  \end{split}
\end{equation*}
Consequently, we have
\begin{equation} \label{Stability-term2}
  \begin{split}
    \sum_{n=1}^N \int_{I_n} &\int_0^t \beta(t-s) a(U_1(t) -U_1(s),A^l U_2(t))\,ds\,dt\\
    &= \frac{1}{2}\sum_{n=2}^N \sum_{j=1}^{n-1}\int_{I_n}\int_{I_j}
          \beta(t-s)\,ds\,dt\, \partial_n \| W_{1,n,j}\|_{l+1}^2\\
    &\quad + \frac{1}{2}\sum_{n=2}^N k_n \sum_{j=1}^{n-1}
          \int_{I_n}\int_{I_j} \beta(t-s)\,ds\,dt\, \| \partial_n W_{1,n,j}\|_{l+1}^2\\
    &\quad -\int_0^T \int_0^t \beta(t-s) a(U_1(t)-U_1(s),A^l P_k f_1)\,ds\,dt,
  \end{split}
\end{equation}
where the second term at the right side is non-negative, since the kernel $\beta$ 
is a decreasing function. 
So we need to show that the first term is also non-negative. 
To this end, denoting
\begin{equation*}
  \beta_{n,j} = \frac{1}{k_n k_j}\int_{I_n}\int_{I_j} \beta(t-s)\,ds\,dt ,
\end{equation*}
we have
\begin{equation*}
  \begin{split}
     \sum_{n=2}^N \sum_{j=1}^{n-1}\int_{I_n}\int_{I_j}
        \beta(t-s)\,ds\,dt\, \partial_n \| W_{1,n,j}\|_{l+1}^2
     &=\sum_{n=2}^N k_n \sum_{j=1}^{n-1} k_j 
          \beta_{n,j}\, \partial_n \| W_{1,n,j}\|_{l+1}^2\\
     &=\sum_{j=1}^{N-1} k_j \sum_{n=j+1}^N k_n 
          \partial_n\{\beta_{n,j} \|W_{1,n,j}\|_{l+1}^2\}\\
     &\quad -\sum_{j=1}^{N-1} k_j \sum_{n=j+1}^N k_n 
          \|W_{1,n-1,j}\|_{l+1}^2 \, \partial_n \beta_{n,j},
  \end{split}
\end{equation*}
where we changed the order of summation and used \eqref{BackwardDifference-p1} 
for the last equality. 
Now it is necessary to show that both terms at the right side are non-negative. 
For the first term  we have
\begin{equation*} 
  \begin{split}
    \sum_{j=1}^{N-1} & k_j \sum_{n=j+1}^N k_n \partial_n
       \{ \beta_{n,j} \|W_{1,n,j}\|_{l+1}^2\}\\
    &=\sum_{j=1}^{N-1} k_j \sum_{n=j+1}^N
       \{\beta_{n,j} \|W_{1,n,j}\|_{l+1}^2 - \beta_{n-1,j} \|W_{1,n-1,j}\|_{l+1}^2\}\\
    &=\sum_{j=1}^{N-1} k_j \{\beta_{N,j}\|W_{1,N,j}\|_{l+1}^2
       -\beta_{j,j} \|W_{1,j,j}\|_{l+1}^2\}\\
    &=\sum_{j=1}^{N-1} k_j \beta_{N,j}\|W_{1,N,j}\|_{l+1}^2>0.
\end{split}
\end{equation*}
For the second term, we should show that $\partial_n \beta_{n,j}<0$. 
Indeed, changing the variable $t=t_{n-1}+k_n \tau$, we have
\begin{equation*}
  \beta_{n,j} = \frac{1}{k_n k_j} \int_{I_n} \int_{I_j} \beta(t-s)\,ds\,dt 
  = \frac{1}{k_j} \int_{I_j} \int_0^1 \beta(t_{n-1}+\tau k_n-s)\,d\tau\,ds,
\end{equation*}
and consequently
\begin{equation*}
  \partial_n \beta_{n,j} = \frac{1}{k_j} \int_{I_j}\int_0^1 
  ( \beta(t_{n-1}+\tau k_n-s)- \beta(t_{n-2}+\tau k_{n-1}-s))\,d\tau\,ds < 0,
\end{equation*}
since the kernel $\beta$ is decreasing. 

Finally, it remains to study the third part on the left side of 
\eqref{Stability-eq2}. Recalling $v_0=U_{2,0}^-$ we have
\begin{equation*} 
  \begin{split}
    \rho \sum_{n=1}^{N-1} ([U_2]_n,A^l U_{2,n}^+) 
       &+\rho (U_{2,0}^+,A^l U_{2,0}^+)
       -\rho (v_0,A^l U_{2,0}^+)\\
    &=\rho \sum_{n=0}^{N-1} ([U_2]_n,A^l U_{2,n}^+)
     =\rho \sum_{n=0}^{N-1} k_{n+1}(\partial_n U_{2,n+1},A^l U_{2,n+1}),
  \end{split}
\end{equation*}
that by \eqref{BackwardDifference-p2} implies
\begin{equation} \label{Stability-term3}
  \begin{split}
    \rho \sum_{n=1}^{N-1} &([U_2]_n,A^l U_{2,n}^+) 
       +\rho (U_{2,0}^+,A^l U_{2,0}^+)
       -\rho (v_0,A^l U_{2,0}^+)\\
    &=\frac{1}{2} \rho \sum_{n=0}^{N-1} 
       \Big\{ k_{n+1} \partial_n(A^{l/2} U_{2,n+1},A^{l/2} U_{2,n+1})\\
    &\qquad +k_{n+1}^2 (\partial_n A^{l/2} U_{2,n+1},\partial_n A^{l/2} U_{2,n+1}) \Big\}\\
    &=\frac{1}{2}\rho (A^{l/2} U_{2,N},A^{l/2} U_{2,N}) 
       - \frac{1}{2}\rho (A^{l/2} U_{2,0},A^{l/2} U_{2,0})\\
    &\qquad +\frac{1}{2}\rho \sum_{n=0}^{N-1} (A^{l/2} [U_2]_n,A^{l/2} [U_2]_n)\\
    &=\frac{1}{2}\rho \| U_{2,N}\|_{l}^2 - \frac{1}{2}\rho \| v_0\|_{l}^2
         +\frac{1}{2}\rho \sum_{n=0}^{N-1} \| [U_2]_n \|_{l}^2.
  \end{split}
\end{equation}

4. Hence, putting \eqref{Stability-term1}, \eqref{Stability-term2} and 
\eqref{Stability-term3} in \eqref{Stability-eq2}, 
we conclude the energy identity \eqref{dG0-StabilityIdentity}. 

5. Finally, we prove the stability estimate \eqref{dG0-StabilityEstimate}. 
Recalling the fact that all terms on the left side of the stability identity 
\eqref{dG0-StabilityIdentity} are non-negative, we have 
\begin{equation}   \label{LastStepInequality}
  \begin{split}
    \eta_N  \| U_{1,N}\|_{l+1}^2 +\rho \| U_{2,N}\|_{l}^2 
      &\leq  \|u_0\|_{l+1}^2 +\rho \|v_0\|_{l}^2 
      +2\int_0^T \big \{ \eta  a(P_k f_1,A^l U_1) + (P_k f_2,A^l U_2)\big\} \, dt\\
    &\quad +2\int_0^T\int_0^t  \beta(t-s) a\Big(P_k f_1(t),A^l(U_1(t)-U_1(s))\Big)\ ds\ dt,
  \end{split}
\end{equation}
Then, using the Cauchy-Schwarz inequality, and the facts that 
$\eta\leq 1$, $\|\beta\|_{L_1(0,\infty)}=\gamma$ and 
\begin{equation*}
  \int_0^T |P_k f|\ dt \leq \int_0^T |f|\ dt,
\end{equation*}
in a classical way,  we conclude the 
stability estimate \eqref{dG0-StabilityEstimate}, 
for some constant $C=C(\gamma,\rho)$. Now the proof is complete. 
\end{proof}

We note that, having more regularity of the solution, see [13], the 
energy identity (4.7) and the stability estimate (4.8) also hold for 
$l \in \mathbb{R}$.

\section{ A priori error estimates}
Here, we prove optimal order a priori error estimates for the displacement 
$u_1=u$ and the velocity $u_2=\dot u$. 

We denote the standard piecewise constant interpolation of  a function $v$ 
with $\tilde v$, corresponding to the partition 
$0 = t_0 < t_1 , \ldots < t_N = T$ of the interval $(0,T)$. 
We also recall the error estimates 
\begin{equation} \label{InterpolationErrorEstimate}
  \int_{I_n} |\tilde v - v|\ dt \leq C k_n\int_{I_n} | \dot v |\ dt. 
\end{equation}

\begin{theorem}
Let $(u_1,u_2)$ and $(U_1,U_2)$ be the solutions of \eqref{weakform2Compact} 
and \eqref{dG0}, respectively. 
Then, with $e=(e_1,e_2)=(U_1,U_2)-(u_1,u_2)$ and 
$C=C(\gamma,\rho)$, we have
\begin{equation} \label{Apriori1}
    \|e_{1,N}\|_1+\|e_{2,N}\|
    \le C \sum_{n=1}^N k_n \int_{I_n} \big\{\| \dot u_2\|_1+\| \dot u_1\|_2\big\}\, dt,\\
\end{equation}
\begin{equation} \label{Apriori2}
  \| e_{1,N} \|
    \le C \sum_{n=1}^N k_n \int_{I_n} \big\{\| \dot u_2\|+\| \dot u_1\|_1\big\}\, dt.
\end{equation}
\end{theorem}
\begin{proof}
We set
\begin{equation*}
  e=(U_1,U_2)-(u_1,u_2)
  =\big((U_1,U_2)-(\tilde u_1,\tilde u_2)\big)+\big((\tilde u_1,\tilde u_2)-(u_1,u_2)\big)
  =\theta + \omega,
\end{equation*}
where $\tilde u_i, i=1,2$, is the standard piecewise constant interpolation 
of $u_i$. 
We can estimate $\omega$ by \eqref{InterpolationErrorEstimate}, 
so we need to find estimates for $\theta$. 
Recalling Galerkin's orthogonality \eqref{GalerkinOrthogonality}, we have
\begin{equation*}
  \begin{split}
    B(\theta,V)&=-B(e,V)-B(\omega,V)=-B(\omega,V)\\
    &=\sum_{n=1}^N \int_{I_n} \Big \{
      -a(\dot \omega_1,V_1)+a(\omega_2,V_1)\\
    &\qquad -\rho (\dot \omega_2,V_2)-a(\omega_1,V_2)
      +\int_0^t \beta(t-s)a(\omega_1(s),V_2(t))\, ds \Big \}\, dt\\
    &\quad -\sum_{n=1}^{N-1}\left \{a([\omega_1]_n,V_{1,n}^+) 
                                                          +\rho ([\omega_2]_n,V_{2,n}^+) \right \}\\
    &\quad - a(\omega_{1,0}^+,V_{1,0}^+) -\rho  (\omega_{2,0}^+,V_{2,n}^+),
  \end{split}
\end{equation*}
and, having the fact that $\omega_i, i=1,2$, vanish at the time nodes and 
$V_i$ are piecewise constant functions, we have 
\begin{equation} \label{ThetaEquation}
  B(\theta,V)=\sum_{n=1}^N \int_{I_n} \Big 
    \{a(\omega_2,V_1)+\Big(-A\big\{\omega_1
       +\int_0^t \beta(t-s) \omega_1(s) \, ds \big\},V_2\Big)\Big\}\, dt.
\end{equation}
Therefore $\theta$ satisfies \eqref{dG0-General} with 
$f_1=\omega_2$ and 
$f_2=-A\big\{\omega_1+\int_0^t \beta(t-s) \omega_1(s) \, ds \big\}$.
Hence, applying the stability estimate \eqref{dG0-StabilityEstimate} and recalling 
 $\theta_{i,0}=\theta_i(0)=0$, we have
\begin{equation} \label{StabilityTheta}
  \begin{split}
    \| \theta_{1,N}\|_{l+1}+\| \theta_{2,N}\|_{l}
  &\leq  C \Big\{
    \| \theta_{1,0}\|_{l+1} + \| \theta_{2,0}\|_{l} 
   +\int_0^T \| \omega_2 \|_{l+1} \,dt\\
  &\qquad + \int_0^T \| A\omega_1\|_l
               +\Big\|\int_0^t \beta(t-s) A\omega_1(s) \, ds \Big\|_{l} \, dt \Big\}\\
  &\leq  C \Big\{
    \int_0^T \| \omega_2 \|_{l+1} \,dt
     + \int_0^T \| A\omega_1\|_l \\
     &\qquad +\Big\|\int_0^t \beta(t-s) A\omega_1(s) \, ds \Big\|_{l} \, dt \Big\}. 
  \end{split}
\end{equation}
Now, we consider two choices $l=0,-1$. 

To prove the first  a priori error estimate \eqref{Apriori1}, we set 
$l=0$. Then, recalling $e=\theta+\omega$ and $\omega_{i,N}=0$, we have
\begin{equation*}
    \| e_{1,N}\|_1+\| e_{2,N}\|
  \leq  C \Big\{
    \int_0^T \| \omega_2 \|_1 \,dt
    +\int_0^T \| A\omega_1\|
    +\Big\|\int_0^t \beta(t-s) A\omega_1(s) \, ds \Big\| \, dt \Big\}. 
\end{equation*}
Now, using \eqref{InterpolationErrorEstimate}, we have 
\begin{equation*}
  \begin{split}
    &\int_0^T \| \omega_2 \|_1 \,dt 
      =\sum_{n=1}^N \int_{I_n} \|\tilde u_2 -u_2\|_1 \, dt 
      \leq C \sum_{n=1}^N k_n\int_{I_n} \|\dot u_2\|_1 \, dt,\\
    &\int_0^T \|A \omega_1\|\, dt 
      =\sum_{n=1}^N \int_{I_n} \| A(\tilde u_1 -u_1)\| \, dt
      \leq C \sum_{n=1}^N k_n\int_{I_n} \| A \dot u_1\| \, dt,\\
    &\int_0^T \Big\| \int_0^t \beta(t-s) A \omega_1(s) \, ds \Big\|\, dt 
      \leq C \int_0^T \int_0^t \beta(t-s)\|A \omega_1(s)\| \, ds \, dt\\
    &\qquad\qquad\qquad\qquad\qquad\qquad\quad\  
      \leq C \int_0^T \beta \, dt \int_0^T \|A \omega_1\|\, dt\\
    &\qquad\qquad\qquad\qquad\qquad\qquad\quad\  
      \leq C \gamma \sum_{n=1}^N \int_{I_n} \| A(\tilde u_1 -u_1)\| \, dt\\
    &\qquad\qquad\qquad\qquad\qquad\qquad\quad\  
      \leq C \sum_{n=1}^N k_n\int_{I_n} \| A\dot u_1\| \, dt,
  \end{split}
\end{equation*}
that, having $\|Av\| \leq \|v\|_2$, implies the first a priori error estimate \eqref{Apriori1}. 

For the second error estimate we choose $l=-1$ in 
\eqref{StabilityTheta}. 
Then, recalling $e=\theta+\omega$ and $\omega_{i,N}=0$, we have
\begin{equation*}
  \begin{split}
    \| e_{1,N}\| &+\| e_{2,N}\|_{-1}\\
    &\leq  C \Big\{
    \int_0^T \| \omega_2 \| \,dt
    +\int_0^T \| A\omega_1\|_{-1}
    +\Big\|\int_0^t \beta(t-s) A\omega_1(s) \, ds \Big\|_{-1} \, dt \Big\}.
  \end{split}
\end{equation*}
Now, using \eqref{InterpolationErrorEstimate}, we have 
\begin{equation*}
  \begin{split}
    &\int_0^T \| \omega_2 \| \,dt 
      =\sum_{n=1}^N \int_{I_n} \|\tilde u_2 -u_2\| \, dt 
      \leq C \sum_{n=1}^N k_n\int_{I_n} \|\dot u_2\| \, dt,\\
    &\int_0^T \| A\omega_1\|_{-1} \, dt 
      =\sum_{n=1}^N \int_{I_n} \| A(\tilde u_1 -u_1)\|_{-1} \, dt
      \leq C \sum_{n=1}^N k_n\int_{I_n} \| A\dot u_1\|_{-1} \, dt,\\
    &\int_0^T \Big\| \int_0^t \beta(t-s) A\omega_1(s) \, ds \Big\|_{-1} \, dt 
      \leq C \int_0^T \int_0^t \beta(t-s)\| A \omega_1(s)\|_{-1} \, ds \, dt\\
    &\qquad\qquad\qquad\qquad\qquad\qquad\quad\  
      \leq C \int_0^T \beta \, dt \int_0^T \|A \omega_1\|_{-1} \, dt\\
    &\qquad\qquad\qquad\qquad\qquad\qquad\quad\  
      \leq C \gamma \sum_{n=1}^N \int_{I_n} \| A(\tilde u_1 -u_1)\|_{-1} \, dt\\
    &\qquad\qquad\qquad\qquad\qquad\qquad\quad\  
      \leq C \sum_{n=1}^N k_n\int_{I_n} \| A\dot u_1\|_{-1} \, dt,
  \end{split}
\end{equation*}
that, having $\|Av\|_{-1} \leq \|v\|_1$, implies the second a priori error estimate 
\eqref{Apriori2}. 
Now the proof is complete.
\end{proof}

\section{Numerical example}
In this section we illustrate that dG(0) method capture the behavior of the 
solution and also its rate of convergence $O(k)$,  
by solving an example for a two dimensional square shape structure. 
We use the finite element method based on continuous 
piecewise linear polynomials for spatial discretization, and 
we use a uniform triangulation with mesh size $h$. 
Here we compute $\omega_{n,j}, \bar f_n$ and $\bar g_n$ 
in \eqref{omega_fn_gn}  using a simple 
quadrature, the midpoint rule. 

We consider the domain be the two dimensional unit square and 
the initial conditions: $u(x,0)=0\,\textrm{m}$,
$\dot{u}(x,0)=0\,\textrm{m/s}$,
the boundary conditions: 
$u=0$ at $x=0$, $g=(0,-1)\,\textrm{Pa}$ at $x=1$ and zero
on the rest of the boundary. The volume load is assumed to be 
$f=0\,\textrm{N/m}^3$. 
The model parameters are: $\gamma=0.5,\ \tau=1,\alpha=2/3 $ and 
$\rho=3000\,\textrm{kg/m}^3$.
The oscillatory behavior of the the solution of the model problem 
is illustrated in Figure 1, for different time steps $k_n=2^{-5}, 2^{-6}$. 

We also verify numerically the temporal rate of convergence $O(k)$ 
for $\|e_{1,N}\|$.
Lacking of an explicit solution we compare with a numerical solution 
with fine mesh sizes $h,\,k$. 
Here we consider 
$h=0.089095,\,k_{\text{min}}=2^{-6}$. 
The result is displayed in Figure 2.

\begin{figure}[htbp]
  \begin{center}
         \includegraphics[width=10cm,height=6cm]
                              {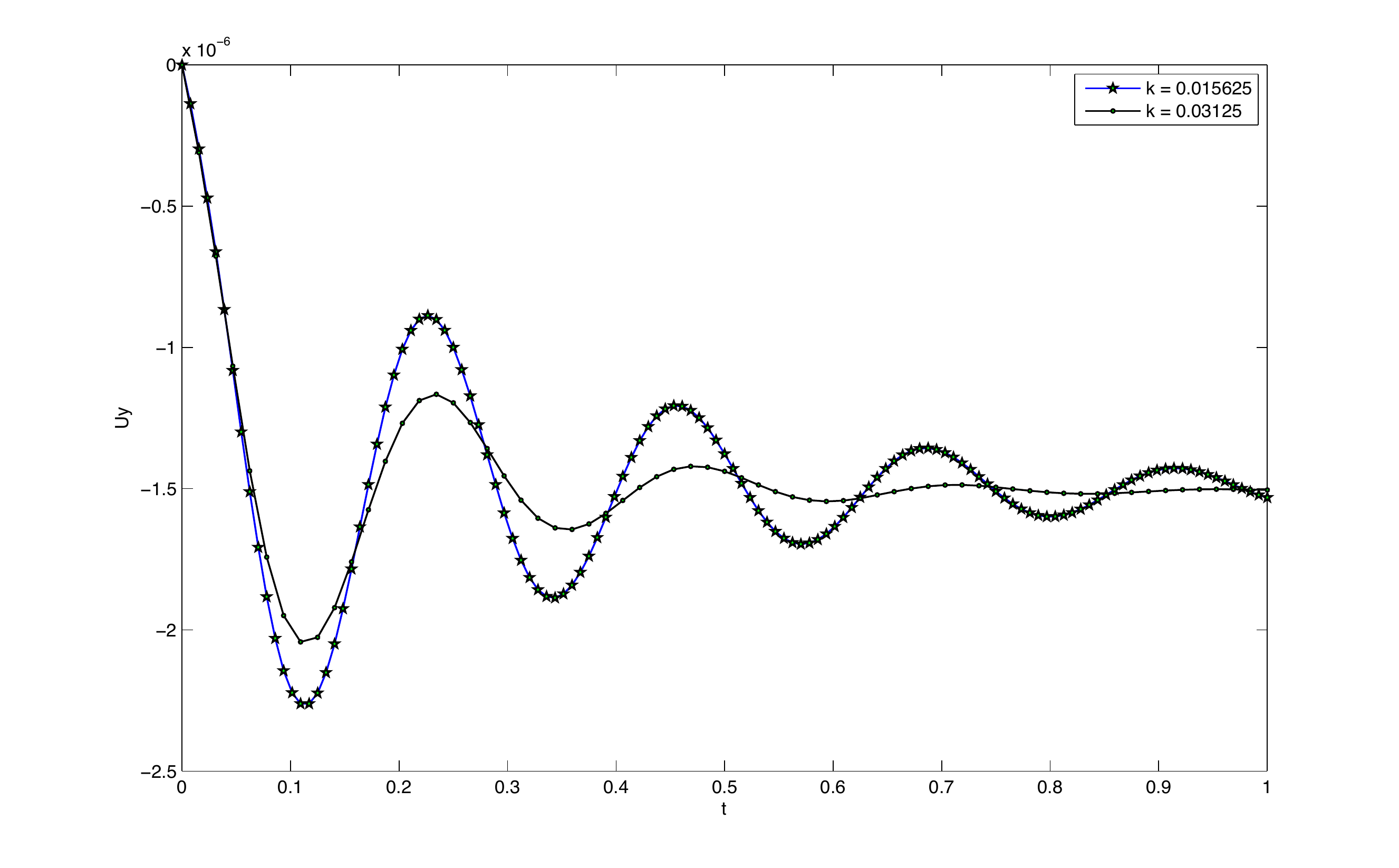}
  \end{center}
       \caption{Oscilatory behavior of point $(1,1)$ of the 2D unit square domain.}
\end{figure}

\begin{figure}[htbp]
  \begin{center}
       \includegraphics[width=10cm,height=6cm]
                              {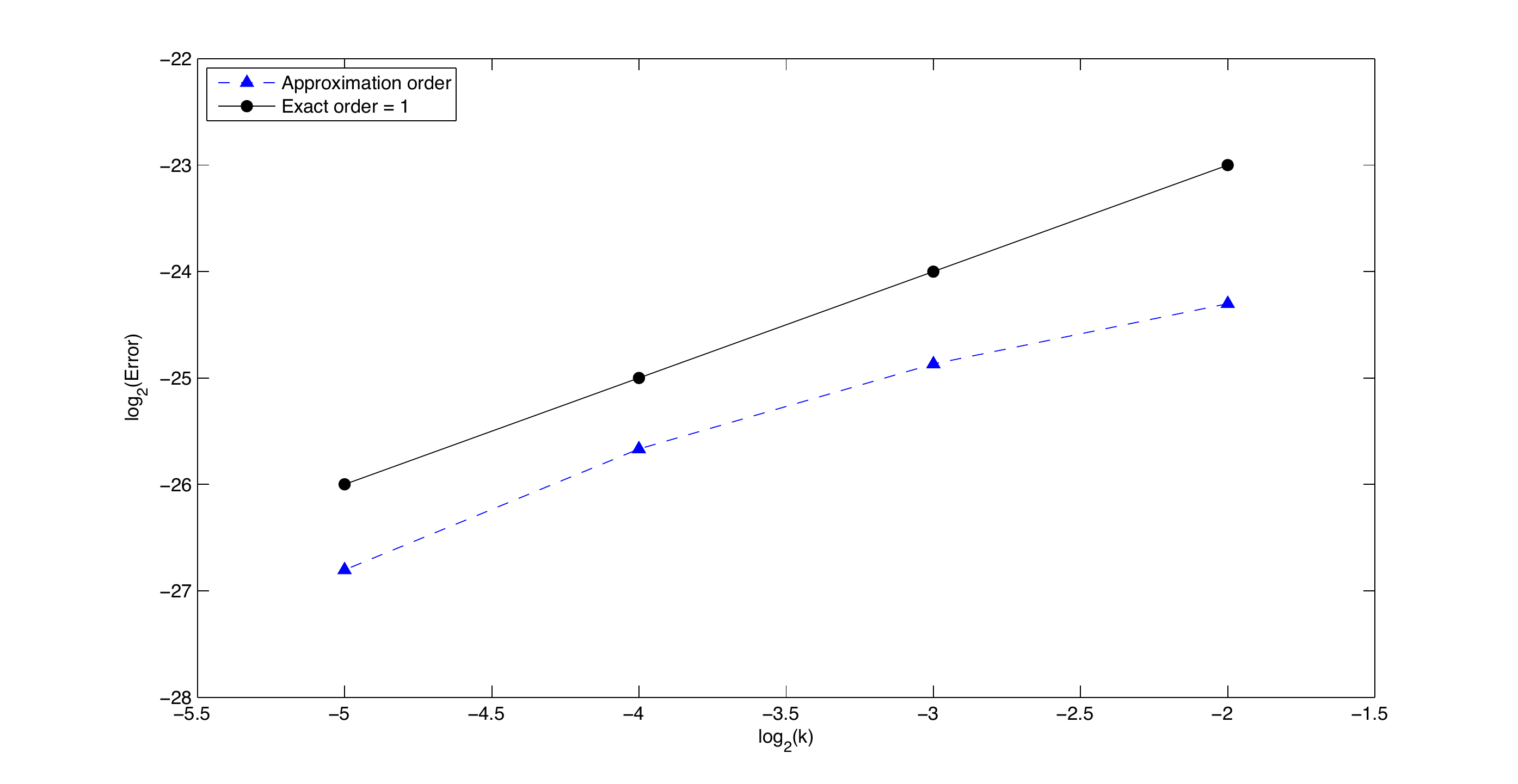}
  \end{center}
       \caption{Rate of convergence of temporal discretization.}
\end{figure}

\newpage


\begin{thebibliography}{10}

\bibitem{AdolfssonEnelundLarsson2003}
K.~Adolfsson, M.~Enelund, and S.~Larsson, \emph{Adaptive discretization of an
  integro-differential equation with a weakly singular convolution kernel},
  Comput. Methods Appl. Mech. Engrg. \textbf{192} (2003), 5285--5304.

\bibitem{AdolfssonEnelundLarsson2008}
K.~Adolfsson, M.~Enelund, and S.~Larsson, \emph{Space-time discretization of an
  integro-differential equation modeling quasi-static fractional-order
  viscoelasticity}, J. Vib. Control \textbf{14} (2008), 1631--1649.

\bibitem{AdolfssonEnelundLarssonRacheva2006}
K.~Adolfsson, M.~Enelund, S.~Larsson, and M.~Racheva, \emph{Discretization of
  integro-differential equations modeling dynamic fractional order
  viscoelasticity}, LNCS \textbf{3743} (2006), 76--83.

\bibitem{BagleyTorvik1983}
R.~L. Bagley and P.~J. Torvik, \emph{Fractional calculus--a different approach
  to the analysis of viscoelastically damped structures}, AIAA J. \textbf{21}
  (1983), 741--748.

\bibitem{StigFardin2010}
S.~Larsson and F.~Saedpanah, \emph{The continuous {G}alerkin method for an
  integro-differential equation modeling dynamic fractional order
  viscoelasticity}, IMA J. Numer. Anal. \textbf{30} (2010), 964--986.

\bibitem{LinThomeeWahlbin1991}
Y.~Lin, V.~Thom\'ee, and L.~B. Wahlbin, \emph{Ritz-volterra projections to
  finite-element spaces and application to integro-differential and related
  equations}, SIAM J. Numer. Anal. \textbf{28} (1991), 1047--1070.

\bibitem{Lubich1988}
C.~Lubich, \emph{Convolution quadrature and discretized operational calculus
  {I}}, Numer. Math. \textbf{52} (1988), 129--145.

\bibitem{McLeanSloanThomee2006}
W.~McLean, I.~H. Sloan, and V.~Thom{\'e}e, \emph{Time discretization via
  {L}aplace transformation of an integro-differential equation of parabolic
  type}, Numer. Math. \textbf{102} (2006), 497--522.

\bibitem{PaniThomeeWahlbin1992}
A.~K. Pani, V.~Thom\'ee, and L.~B. Wahlbin, \emph{Numerical methods for
  hyperbolic and parabolic integro-differential equations}, J. Integral
  Equations Appl. \textbf{4} (1992), 533--584.

\bibitem{RiviereShawWhiteman2007}
B.~Rivi\`ere, S.~Shaw, and J.~R. Whiteman, \emph{Discontinuous {G}alerkin
  finite element methods for dynamic linear solid viscoelasticity problems},
  Numer. Methods Partial Differential Equations \textbf{23} (2007), 1149--1166.

\bibitem{FardinBIT2013}
F.~Saedpanah, \emph{A posteriori error analysis for a continuous space-time
  finite element method for a hyperbolic integro-differential equation}, BIT
  Numer. Math. \textbf{53} (2013), 689--716.

\bibitem{FardinIMAJNUMANAL2014}
\bysame, \emph{Continuous {G}alerkin finite element methods for hyperbolic
  integro-differential equations}, IMA J. Numer. Anal. (2014), doi:
  10.1093/imanum/dru024.

\bibitem{FardinEJM2014}
\bysame, \emph{Well-posedness of an integro-differential equation with positive
  type kernels modeling fractional order viscoelasticity}, European J. Mech.-A
  Solid \textbf{44} (2014), 201--211.

\bibitem{SchadleLopezLubich2008}
A.~Sch\"adle, M.~L\'opez-Fern\'andez, and Ch. Lubich, \emph{Adaptive, fast, and
  oblivious convolution in evolution equations with memory}, SIAM J. Sci.
  Comput. \textbf{30} (2008), 1015--1037.

\bibitem{ShawWhiteman2004}
S.~Shaw and J.~R. Whiteman, \emph{A posteriori error estimates for space-time
  finite element approximation of quasistatic hereditary linear viscoelasticity
  problems}, Comput. Methods Appl. Mech. Engrg. \textbf{193} (2004),
  5551--5572.

\bibitem{SloanThomee1986}
I.~H. Sloan and V.~Thom\'ee, \emph{Time discretization of an
  integro-differential equation of parabolic type}, SIAM J. Numer. Anal.
  \textbf{23} (1986), 1052--1061.

\bibitem{Thomee_Book}
V.~Thom{\'e}e, \emph{Galerkin {F}inite {E}lement {M}ethods for {P}arabolic
  {P}roblems}, second ed., Springer Series in Computational Mathematics,
  vol.~25, Springer-Verlag, 2006.

\end{thebibliography}

\providecommand{\bysame}{\leavevmode\hbox to3em{\hrulefill}\thinspace}
\providecommand{\MR}{\relax\ifhmode\unskip\space\fi MR }
\providecommand{\MRhref}[2]{%
  \href{http://www.ams.org/mathscinet-getitem?mr=#1}{#2}
}
\providecommand{\href}[2]{#2}

\end{document}